\documentclass[11pt,a4paper]{article}
\usepackage{fancyhdr}
\usepackage{amsmath}
\usepackage{amsthm}
\usepackage{amssymb}
\usepackage{wasysym}
\usepackage{paralist}
\usepackage{makeidx}
\usepackage[all]{xy}
\usepackage{mathdots}
\usepackage{yhmath}
\usepackage{mathtools}

\usepackage{young}
\usepackage[vcentermath]{youngtab}

\usepackage{graphicx}
\usepackage{epstopdf}

\input xy

\newcommand{\nc}{\newcommand}
 
 \nc{\cl}{\centerline}
 
 \nc{\SL}{{\rm SL}}
 \nc{\hatQ}{{\hat Q}}
 \nc{\sgn}{{\rm sgn}}
 
 \nc{\seee}{\mathbb C}
 
 \nc{\varleq}{\preccurlyeq}
 
 \newcommand{\id}{{\rm id}}

 \nc{\hatlambda}{{\hat\lambda}}
 
 \nc{\daggerlambda}{{\lambda^\dagger}}

    \nc{\barr}{{\bar r}}
  \nc{\bart}{{\bar t}}
  
    \nc{\barsigma}{{\bar \sigma}}

\nc\diag{{\rm diag}}
\renewcommand{\vert}{{\,|\,}}
\nc{\hatL}{{\hat L}}

\nc{\barE}{{\bar   E}}
\nc{\D}{{\mathcal D}}
\nc{\E}{{\mathcal E}}
\nc{\F}{{\mathcal F}}
\nc{\FF}{{\mathcal F}}
\nc{\I}{{\mathcal I}}
\nc{\even}{{\rm e}}
\nc{\ep}{\epsilon}
\nc{\odd}{{\rm o}}
\nc{\Coker}{{\rm Coker}}
\nc{\olE}{{\overline E}}
\nc{\indBG}{{\rm ind}_B^G\,}
\nc{\indHG}{{\rm ind}_H^G\,}

\nc{\que}{{\mathbb Q}}
\nc{\barlambda}{{\bar\lambda}}
\nc{\barmu}{{\bar\mu}}
\nc{\barnu}{{\bar\nu}}
\nc{\bartau}{{\bar\tau}}
\nc{\barm}{{\bar m}}
\nc{\divind}{{\rm div.ind}}
\nc{\tl}{{\tilde{\lambda}}}
\nc{\dar}{\downarrow}

\nc{\eno}{{\mathbb N}_0}

\nc{\Sym}{{\rm Sym}}
\nc{\Symm}{{\rm Sym}}

\newcommand{\q}{\quad}
\newcommand{\de}{\delta}

\renewcommand{\mod}{{\rm mod}}

\newcommand{\bs}{\bigskip}

\renewcommand{\vert}{\,|\,}

\renewcommand{\sgn}{{\rm sgn}}

\xyoption{all}

\renewcommand{\vert}{\,|\,}
 
\newcommand{\zed}{{\mathbb Z}}

\newcommand{\Hom}{{\rm Hom}}

\renewcommand{\mod}{{\rm mod}}

\newcommand{\GL}{{\rm GL}}

\renewcommand{\mod}{{\rm{mod}}}

\nc{\geom}{{\rm geom}}
\nc{\rep}{{\rm rep}}

\newcommand{\ch}{{\rm ch\,}}

\newtheorem{definition}{Definition}[section]
\newtheorem{proposition}[definition]{Proposition}
\newtheorem{theorem}[definition]{Theorem}
\newtheorem{lemma}[definition]{Lemma}
\newtheorem{corollary}[definition]{Corollary}
\newtheorem{example}[definition]{Example}

\newtheorem{remark}[definition]{Remark}

\begin{document}

\centerline{\bf A Clebsch-Gordan decomposition  in positive characteristic}

\bigskip

\centerline{Stephen Donkin and Samuel Martin}

\bigskip

{\it Department of Mathematics, University of York, York YO10 5DD\\  \\
and
\\
Earlham Institute,
Norwich Research Park,
Norwich,
NR4 7UZ,
UK
}\\

\medskip

\centerline{\tt stephen.donkin@york.ac.uk \hskip 50pt  samuel.martin@earlham.ac.uk}

\bs

\centerline{3  April    2018}

\bs\bs\bs

\section*{Abstract}   Let $G$ be the special linear group of degree $2$ over an algebraically closed field $K$. Let $E$ be the natural module and $S^rE$ the $r$th symmetric power.  We consider here, for $r,s\geq  0$, the tensor product of $S^rE$ and the dual of $S^sE$. In characteristic zero this tensor product decomposes according to the Clebsch-Gordan formula. We consider here the situation when $K$ is a field of positive characteristic.  We show that each  indecomposable component occurs with multiplicity one and identify which modules  occur  for given $r$ and $s$.

\section*{Introduction}

\q Let $K$ be an algebraically closed field of characteristic $p>0$.  Let $G$ be the  special linear group over $K$ of degree $2$, regarded as a linear algebraic group.

\q Weights and weight spaces will be computed with respect to the maximal torus $T$ of $G$ consisting of diagonal matrices. More precisely, given a rational $T$-module $V$ we have the weight space decomposition $V=\bigoplus_{r\in \zed} V^r$, where 
$$V^r=\{v\in V \vert 
\left(\begin{matrix}
t&0\cr
0&t^{-1}
\end{matrix}
\right) =t^r v \  \hbox{ for all } 0\neq t\in K
\}.$$

The space $V^r$ is the $r$-weight space, its dimension is the multiplicity of $r$ as a weight of $V$ and its elements are the vectors of weight $r$.
  The character of a finite dimensional rational $T$-module $V$ is the Laurent polynomial $\sum_{r\in \zed} (\dim V^r) x^r$.

\q Let $E$ be the natural  $G$-module (of column vectors of length $2$).   For $r\geq 0$, we have the $r$th symmetric power $S^rE$, which we also denote $\nabla(r)$. The module $\nabla(r)$ has an interpretation as an induced module, in the sense of algebraic group theory. The dual module $\Delta(r)$ is the corresponding Weyl module.  The $G$-socle $L(r)$ of $\nabla(r)$ is a simple module with highest weight $r$,  and indeed the modules $L(r)$, $r\geq 0$, form a complete set of pairwise non-isomorphic simple rational $G$-modules.   

\q By a good filtration of a rational $G$-module $V$ we mean a filtration $0=V_0\leq V_1\leq \cdots \leq V$ such that $V=\bigcup_{i=0}^\infty V_i$ and, for each $i>0$,  the module $V_i/V_{i-1}$ is either $0$ or isomorphic to $\nabla(r_i)$, for some $r_i\geq 0$. For a finite dimensional rational $G$-module with a good filtration and $r\geq 0$,  the cardinality of the set $\{i>0 \vert V_i/V_{i-1} \cong \nabla(r)\}$ is independent of the choice  of good filtration and we denote it $(V:\nabla(r))$.

\q The purpose of this paper is to identify the  indecomposable summands  of $\nabla(r)\otimes \Delta(s)$, for $r,s\geq 0$  and to identify which summands occur for given $r$ and $s$.   By duality it is enough to consider the case $r\geq s$.  We introduce our key notion.

\begin{definition}  An indecomposable summand of a $G$-module $\nabla(r)\otimes \Delta(s)$, with $r\geq s\geq 0$, will be called an indecomposable Clebsch-Gordan module.  An arbitrary finite dimensional rational $G$-module will be a called a Clebsch-Gordan module if each indecomposable summand is a Clebsch-Gordan module.
\end{definition}

\q  A finite dimensional rational $G$-module $V$ such that both $V$ and its dual admit a good filtration is called a tilting module. For each $m\geq 0$ there is an indecomposable tilting module $T(m)$ such that $m$ is the highest weight of $T(m)$ and occurs with multiplicity one. The modules $T(m)$, $m\geq 0$, form a complete set of pairwise non-isomorphic indecomposable tilting modules.

\q We shall show that each tilting module $T(m)$ is a Clebsch-Gordan module. Furthermore, we obtain all indecomposable Clebsh-Gordan modules from the tilting modules in the following way. We  write $m$ in the form $m=(p^N-1)+\sigma$, where $p^N-1 \leq m< p^{N+1}-1$. We write $\sigma$ in its base $p$ expansion $\sigma=\sum_{i=0}^N p^i\sigma_i$ and set $S(m)=\{ i \vert 0\leq i <N, \sigma_i\neq 0\}$.  For each subset $I$ of $S(m)$ 
we define a quotient $T(m)_I$ of $T(m)$. We show that the modules $T(m)_I$, as $m$ and $I$ vary, form a complete set of pairwise non-isomorphic indecomposable Clebsch-Gordan modules. We show that  the multiplicity of an indecomposable summand of a module of the form $\nabla(r)\otimes \Delta(s)$ is at most one and explicitly describe when a module $T(m)_I$ appears as a summand.

\q In  the case $r=s$, and $p>2$,  the condition for $\nabla(2)=\Delta(2)$ to be a summand of $\nabla(r)\otimes \Delta(s)$ was obtained by Goodbourn, \cite[Theorem 4.8]{Oliver}.

\q Two other versions  of the \lq\lq Clebsch-Gordan problem" are available. In \cite{DoHe}, Doty and Henke  give a decomposition of the tensor product of simple modules $L(r)\otimes L(s)$, as a direct sum of indecomposable modules (which are \lq\lq twisted" tilting modules, cf. \cite{DattDonkin}).  In \cite{Cavallin},  Cavallin describes  a decomposition of the ${\rm GL}_2(K)$-module $S^rE\otimes S^sE$,  as a direct sum of indecomposable modules (taking advantage of the fact that $S^rE\otimes S^sE$ is injective in the polynomial category).

\q For terminology and background results not explained here the reader may consult the book by Jantzen, \cite{RAG}.

\section{Decomposing the tilting modules $Y(r)$}

\q We write $\chi(r)$ for the character of $\nabla(r)$, $r\geq 0$.    We note that an indecomposable summand occurs at most once in our modules of interest.  For finite dimensional rational modules $V,W$ with $W$ indecomposable we write $(V\vert W)$ for the multiplicity of $W$ as a summand of $V$.

\begin{remark} Let $V$ be an indecomposable module.  For  $r\geq s\geq 0$, the multiplicity $(\nabla(r)\otimes \Delta(s) \vert V)$ is at most one.  This may be seen in the following way. The module $\nabla(r)\otimes \Delta(s)$ has a good filtration by \cite[Lemma 3.3]{Alison}.  The character $\chi(r)\chi(s)$ of $\nabla(r)\otimes \Delta(s)$ is,  according to  the usual Clebsch-Gordan formula, $\sum_{0\leq i\leq s} \chi(r+s-2i)$. Hence we have $(\nabla(r)\otimes \Delta(s): \nabla(t))\leq 1$, for $t\geq 0$.  If $(\nabla(r)\otimes \Delta(s) \vert V)\geq 2$ then $\nabla(r)\otimes \Delta(s)$ is isomorphic to $V\oplus V\oplus V'$ for some module $V'$. Choosing $t$ such that $(V:\nabla(t))> 0$   and we get  
$$(\nabla(r)\otimes \Delta(s):\nabla(t))=2(V:\nabla(t))+(V':\nabla(t))\geq 2$$
 a contradiction.
\end{remark}

\q For a non-negative integer $r$ we set 

$$Y(r)=\begin{cases}\nabla(m)\otimes \Delta(m),   & {\rm if}\   r=2m \hbox{ is even};\cr
\nabla(m+1)\otimes \Delta(m) ,  & {\rm if}\   r=2m+1 \hbox{ is odd.}
\end{cases}$$

\q The module $Y(r)$ is,    a tilting module, e.g., by \cite[Lemma 1.2]{Sam}.    Note that the character of $Y(r)$ is given by 
$$\ch Y(r)=\begin{cases}\chi(r)+\chi(r-2)+\cdots+\chi(0), & \hbox{ if } r \hbox{ is even};\cr
\chi(r)+\chi(r-2)+\cdots+\chi(1), & \hbox{ if } r \hbox{ is odd.}   
\end{cases}
\eqno{(*)}
$$

\begin{remark} The tilting module $T(r)$ with highest weight $r$ appears as a summand of $Y(r)$   and hence every   tilting module is a Clebsch-Gordan module.
\end{remark}

\q  A precise description of the non-negative integers $r,s$ such that $\nabla(r)\otimes \Delta(s)$ is tilting is to be found in \cite{Sam}.

\q In this section we determine a decomposition of $Y(r)$ as a direct sum of indecomposable modules, i.e., we determine, for $r,s\geq 0$, when  $(Y(r)\vert T(s))$ is non-zero (and hence $1$).  This generalises the result of Goodbourn  describing the condition for  $T(2)$ to occur as a summand of $Y(r)$, for $p$ odd,   \cite[Theorem 4.8]{Oliver},  which has special significance for the theory of reductive pairs.

\q We shall use the notion of an admissible quadruple  to describe a direct sum decomposition of the tilting modules $Y(r)$, $r\geq 0$.

\q Our method is essentially  to consider the characters  of the modules $T(s)$ and express the character of $Y(r)$ in terms of these. Since the character of $T(s)$ is known (see   Proposition 1.3  below) and the character of $Y(r)$ is given by the usual Clebsch-Gordan formula we may obtain the result by inverting a matrix of $\nabla$-multiplicities in tilting modules.

\q We start with some notation. We write $\eno$ for the set of non-negative integers.   We write the base $p$ expansion of $\sigma \in \eno$ as $\sigma=\sum_{i\geq 0} p^i\sigma_i$ (with $0\leq \sigma_i\leq p-1$ and $\sigma_i=0$ for $i$ large)   or just $\sigma=\sum_{i=0}^N p^i\sigma_i$, if $\sigma < p^{N+1}$.  For a set $I$ of non-negative integers we define $\sigma_I=\sum_{i\in I}p^i\sigma_i$.

\q An element $r\in \eno$ determines  non-negative integers $N$ and $\sigma$ such that $p^N-1\leq r< p^{N+1}-1$ and 
$$r=(p^N-1)+\sigma \eqno(*).$$
We shall say that $(*)$ is the {\em standard expression}  for $r$.

\q We shall need the multiplicities of the module $\nabla(s)$ as a section in a good filtration of the module $T(r)$, for $r,s\geq 0$.  By  taking $q=1$ and restricting to ${\rm SL}_2(K)$ in \cite[3.4(3)]{q-Donk},  we obtain the following.

\begin{proposition} Let $r\in \eno$ with standard expression $r=(p^N-1)+\sigma$.   For $s\in \eno$ we have $(T(r):\nabla(s))\leq 1 $ and 
$(T(r):\nabla(s))\neq  0$ if and only if $s=r-2\sigma_I$ for some subset $I$ of $\{i\in \eno \vert    i\neq N,  \sigma_i\neq 0\}$.

\end{proposition}

\begin{definition} By an admissible triple   we mean a triple    $(N,\sigma,\delta)$  of non-negative integers $N$, $\sigma$ and $\delta$ such that $\sigma,\delta < p^{N+1}-p^N$, such that $\sigma_i+\de_i\leq p-1$ for $0\leq i<N$ and such that  $\sigma_N+\de_N<p-1$. By an admissible quadruple we mean a quadruple   $(N,\sigma,\delta,I)$, where $(N,\sigma,\de)$ is an admissible triple and $I$ is a subset of $\{i\in \eno \vert \sigma_i\neq 0, i\neq N\}$
\end{definition}

\begin{definition}  Let $r,t\in \eno$ with $r\geq t$ and $r-t$ even.   We say that an admissible quadruple  $(N,\sigma,\delta,I)$ is an admissible quadruple for the pair  $(r,t)$ if 
$$t=(p^N-1)+\sigma- 2\sigma_I \eqno{(1)}$$
and 
$$r\equiv  (p^N-1)+\sigma+2\delta  \hskip 10pt   (\mod \  2p^{N+1}) \eqno{(2).}$$

 \q Less formally we shall  say that $(N,\sigma,\delta,I)$ is a solution for $(r,t)$.

\end{definition}

\begin{remark} We note that if $(N,\sigma,\de,I)$ is an  admissible quadruple for $(r,t)$  then,   by (1), $\sigma_I$, and hence $I$, is determined by the triple $(N,\sigma,\de)$.
\end{remark}

\q Our result on the decomposition of the tilting modules $Y(r)$ is the following.

\begin{theorem} Let $r,s\in \eno$. Then $(Y(r)\vert T(s))\neq 0$ if and only if there is an admissible triple  $(N,\sigma,\de)$ such that 
\begin{align*}r&\equiv (p^N-1)+\sigma+2\de \hskip 20pt \hbox{\rm mod}\ 2p^{N+1},   \hskip 10pt  {\rm and}\cr
s&=(p^N-1)+\sigma.
\end{align*}
\end{theorem}

\q Note that in the above statement we have 
\begin{align*}s&=(p^N-1)+\sum_{i=0}^{N-1}p^i\sigma_i+p^N\sigma_N\cr
&\leq (p^N-1)+\sum_{i=0}^{N-1} p^i(p-1)+(p-2)p^N\cr
&=(p^N-1)+(p^N-1)+(p-2)p^N=p^{N+1}-2
\end{align*}
so that $(p^N-1)\leq s < p^{N+1}-1$ and the expression $s=(p^N-1)+\sigma$ in the Theorem is the standard expression.  Now we put $u=(r-s)/2$ and then the first condition in the theorem becomes simply $u_i=\de_i$, for $0\leq i\leq N$.  Hence we may express the Theorem in the following more usable form.

\bs

${\bf \, Theorem\ 1.7}^\prime$  \sl Let $r,s\in \eno$ with $r\geq s$ and $r-s$ even.   Put $u=(r-s)/2$ and express $s$  in standard form $s=(p^N-1)+\sigma$.  Then $(Y(r)\vert T(s))\neq 0$ if and only $\sigma_i+u_i\leq p-1$ for $0\leq i<n$ and $\sigma_N+u_N <  p-1$.

\rm

\bs

\q Our proof of the Theorem  is based on the following existence and uniqueness property.

\begin{proposition} For $r,t\in \eno$ with $r\geq t$ and $r-t$ even there exists a unique admissible quadruple  for $(r,t)$.
\end{proposition}

\q Given the above proposition the theorem follows in a straightforward manner.

\bs

\it Proof of  Theorem 1.7. \rm  

\medskip

\q Let $n$ be a non-negative integer.  We define the matrix $A=(a_{rs})_{0\leq r,s\leq n}$, where $a_{rs}=1$ if $r\geq s$, $r-s$ is even and there exists an admissible triple  $(N,\sigma,\de)$ with 
\begin{align*}r&\equiv (p^N-1)+\sigma+2\de \hskip 20pt \hbox{\rm mod}\ 2p^{N+1},   \cr
s&=(p^N-1)+\sigma
\end{align*}
and $a_{rs}=0$ otherwise. We define  the matrix $B=(b_{rs})_{0\leq r,s\leq n}$, where $b_{rs}=(T(r):\nabla(s))$. We consider the product matrix $AB=C=(c_{rs})_{0\leq r,s\leq n}$. 

\q We  have $c_{rt}=\sum_{s=0}^n a_{rs}b_{st}$, for $0\leq r,t\leq n$. Suppose $a_{rs}b_{st}\neq 0$. Then there exists an admissible triple $(N,\sigma,\de)$ such that 
\begin{align*}r&\equiv (p^N-1)+\sigma+2\de \hskip 20pt \hbox{\rm mod}\ 2p^{N+1},   \cr
s&=(p^N-1)+\sigma.
\end{align*}

\q Moreover, by Proposition 1.3, we have $t=(p^N-1)+\sigma-2\sigma_I$, for some subset $I$ of $\{i\in \eno \vert    i\neq N,  \sigma_i\neq 0\}$.  But then $(N,\sigma,\de,I)$ is an admissible quadruple for $(r,t)$.  Hence the value of $s$ is uniquely determined by this quadruple, by Proposition 1.8.  In particular there is exactly one such $s$ and we have 
$$c_{rt}=\sum_{s=0}^n a_{rs}b_{st}=\begin{cases}1, &\hbox{ if } r\geq t \hbox{ and } r-t \hbox{ is even};\cr
0, & \hbox{otherwise.}
\end{cases}$$

 Thus we have $c_{rt}=(Y(r):\nabla(t))$. 
 
 \q Note that the matrices $A,B,C$ are invertible.     Let $A'=(a_{rs}')_{0\leq r,s\leq n}$, where $a_{rs}'=(Y(r)\vert T(s))$.  The $(r,t)$ entry of $A'B$ is
\begin{align*}\sum_{s=0}^n (Y(r)\,|\,T(s))(T(s):\nabla(t))=(Y(r):\nabla(t))=c_{rt}.
\end{align*}

\q Hence we have $A'B=C=AB$ and therefore $A'=A$.  Hence $(Y(r): T(s))=a_{rs}$, as required.
\qed

\bs

\q  The remainder of this section is devoted to a proof of  Proposition 1.8.   In the analysis that follows we shall assume that $p$ is odd. We  leave  verifications in case $p=2$ to the interested reader. Note that if $(N,\sigma,\de,I)$ satisfies  (1) and the condition $r\equiv (p^N-1)+\sigma+2\de$ (mod $p^{N+1}$) then we have 
\begin{align*}r-((p^N-1)+\sigma+2\de)&\equiv r-((p^N-1)+\sigma-2\sigma_I) \hskip 10pt \hbox{(mod $2$)}\cr
&\equiv r-t\equiv 0   \hskip 20pt  \hbox{(mod $2$).}
\end{align*}
So  in fact we have the condition (2). Thus in what follows it is enough to work with (2) in the simplified form
$$r\equiv  (p^N-1)+\sigma+2\delta  \hskip 20pt   (\mod \  p^{N+1}).$$

We separate  out the cases in which $N=0$.

\bs

\begin{lemma} Suppose $r,t\in \eno$, $r\geq t$ with $r-t$ even. We define $u=(r-t)/2$.   If there exits an admissible quadruple  $(N,\sigma,\delta,I)$ for $(r,t)$  with $N=0$ then $t+u_0< p-1$.  Conversely,   if $t+u_0< p-1$  then there is a unique admissible quadruple     for $(r,t)$,  namely  $(0,t,u_0,\emptyset)$. 
\end{lemma}

\begin{proof} Suppose $(0,\sigma,\de,I)$ is a solution for $(r,t)$. Then $I=\emptyset$, $\sigma=\sigma_0$, $\de=\de_0$ and $t=\sigma_0$ from (1).  Moreover, (2) gives $u_0=\de_0$ so that $t+u_0=\sigma_0+\de_0<p-1$.

\q Now suppose that $t+u_0<p-1$. Then $(0,\sigma,\de,I)$ is a solution if and only if $I=\emptyset$, $t=\sigma=\sigma_0$ and $r=\sigma-2\de_0$ (mod $p$), so $\de_0=u_0$. Hence $(0,t,u_0,\emptyset)$ is the unique solution with $N=0$.

\q Now suppose, for a contradiction, that we have a solution $(N,\sigma,\de,I)$, with $N>0$. From (1) we get
$$1+t+\sum_{i\in I} p^i\sigma_i=p^N+\sum_{i\in {\bar I}} p^i\sigma_i$$
where ${\bar I}$ is the complement of $I$ in $\{0,1,\ldots,N\}$. We must have $0\in I$, for otherwise the left hand side is less than $p^N$. Taking the equation modulo $p$ gives $1+t+\sigma_0\equiv 0$ and hence $1+t+\sigma_0=p$.  From (2) we get $r\equiv -1-1-t+2\de_0$ (mod  $p$), and subtracting $t$ we get $2u\equiv -2-2t+2\de_0$ and hence $u_0\equiv -1-t+\de_0$ (mod  $p$) and hence $\de_0=u_0+1+t$. But now $\sigma_0+\de_0=p-1-t+u_0+1+t=u_0+p$ and the requirement $\sigma_0+\de_0\leq p-1$ is not satisfied.
\end{proof}

\it Proof of  Proposition 1.8.

\rm
\bs

\q Suppose the result is false and that $r+t$ is minimal among all cases in which either uniqueness or existence fails.  The proof is divided into several cases. We investigate a  possible solution $(N,\sigma,\delta,I)$ with $N>0$, so that   (1) and (2) are satisfied with a subset $I$ of $\{0,\ldots,N-1\}$.  By Lemma 1.9, we only need to consider cases in which $N>0$. We write $\sigma=\sigma_0+p\sigma'$, $\delta=\delta_0+p\delta'$. Then we have $\sigma_I=\ep_I\sigma_o+p\sigma'_{I'}$, where $I'$ is the subset $\{i-1\vert 0\neq i\in I\}$ of $\{0,\ldots,N-2\}$ and where $\ep_I=1$ if $0\in I$ and  $\ep_I=0$ if not.  We shall see that $\ep_I,\sigma_0,\de_0$ are determined by (1) and (2) and that $(N,\sigma,\de,I)$ is a solution for $(r,t)$ if and only if $(N-1,\sigma',\de',I')$ is an admissible quadruple  for the  smaller pair $(r',t')$, where $r = r_0 + pr'$ and $t = t_0 + pt'$. The result then follows by existence and uniqueness of a quadruple  for $(r',t')$.

\bs

Case A:  \    $t_0=p-1$.   

\medskip

\q   From (1) we have $-1\equiv -1+\sigma-2\sigma_I$ (mod $p$) and therefore $\sigma_0=0$. Furthermore we have 
$$t =pt'+ p-1=(p^N-1)+\sigma-2\sigma_I$$
so $pt'=p^N-p+p\sigma'-2p\sigma'_{I'}$, i.e.,
$$t'=p^{N-1}-1+\sigma'-2\sigma'_{I'}.\eqno{(3)}$$
From (2) we have
$r\equiv -1+\sigma_0+2\de_0$ (\mod  \, $p$),  i.e., $2\de_0\equiv r+1$ (\mod  \, $p$).  We have three cases to consider: (i) $r_0=p-1$; (ii) $r_0$ odd; and (iii) $r_0$ even, $r_0\neq p-1$.

\medskip

(i) If $r_0=p-1$ then   $\de_0=0$. From (2) we have 
$$r\equiv (p^N-1)+\sigma+2\delta \hskip 20pt ({\rm mod} \, p^{N+1})$$
so we have $pr'\equiv p^N-p+p\sigma'+2p\delta' \ ({\rm mod} \, p^{N+1})$  i.e., 
$$r'\equiv p^{N-1}-1+\sigma'+2\delta'  \hskip 20pt ({\rm mod} \, p^N).\eqno{(4)}$$

\q By minimality (3) and (4) have a unique solution $(N-1,\sigma',\delta',I')$ for $(r',t')$,  and then $(N,\sigma_0+p\sigma',\delta_0+p\delta',I)$ is the unique solution to (1),(2) (where $I=\{i+1\vert i\in I'\}$).  So $(r,t)$ is not a minimal counterexample.

\medskip

(ii) Suppose $r_0$ is odd.  Then $r\equiv -1+2\delta_0$ (\mod \, $p$) gives $\delta_0=(r_0+1)/2$.  From (2) we get 
$$r_0+pr'\equiv (p^N-1)+p\sigma'+(r_0+1)+2p\delta' \hskip 10pt   ({\rm mod} \, p^{N+1})$$ 
 i.e., $pr'\equiv p^N+p\sigma'+ 2p\delta'  \  ({\rm mod} \, p^{N+1})$, or
$$r'-1\equiv (p^{N-1}-1)+\sigma'+2\delta' \hskip 20pt ({\rm mod} \, p^N). \eqno{(5)}$$
\q Since $r_0$ is odd and $r$ and $t$ have the same parity, $r'$ and $t'$ have different parities. Hence $r'-1\geq t'$ and by minimality we have a unique solution  $(N-1,\sigma',\delta',I')$ giving  a solution for $(r'-1,t')$ and then (reversing the steps)  $(N,\sigma_0+p\sigma',\delta_0+p\delta',I)$ is the unique solution for $(r,t)$. Hence $(r,t)$ is not a minimal counterexample.

\medskip

(iii) Suppose $r_0$ is even, $r_0\neq p-1$.  Then $r\equiv -1+2\de_0$ (\mod \, $p$)  gives $\de_0=(p+1+r_0)/2$.
From (2) we get 
$$r_0+pr'\equiv (p^N -1) +p\sigma'+ (p+1+r_0)+2p\delta'  \  \ ({\rm mod}\, p^{N+1})$$  
 i.e.,  $pr'\equiv  p^N+p\sigma'+p+2p\delta' \ ({\rm mod}  \, p^{N+1})$, or
$$r'-2\equiv (p^{N-1}-1)+\sigma'+2\delta' \hskip 20pt ({\rm mod}  \, p^N). \eqno{(6)}$$

\q Note that $r'$ and $t'$ have the same parity and $r>t$ so that $r'-2\geq t'$. By minimality there is a unique admissible quadruple $(N-1,\sigma',\delta',I')$ satisfying (3) and (6) and hence $(N,\sigma_0+p\sigma',\delta_0+p\delta',I)$ is the unique admissible quadruple  satisfying (1) and (2). Hence $(r,t)$ is not a minimal counterexample.

\medskip

\q We assume in the remaining cases that $t_0\neq p-1$. From (1), in any solution $(N,\sigma,\delta,I)$ with $N>0$   we have $t_0\equiv -1+\sigma-2\sigma_I$ 
(\mod \, $p$), giving two possibilities for $\sigma_0$. Either $0\in I$ so that $\sigma_0=p-1-t_0$, or $0\not\in I$ and $\sigma_0=t_0+1$.

\bs\bs

Case B: \q $t_0\neq p-1$ and $r_0,t_0$  have the same parity (hence $r'$ and $t'$ have the same parity) and $r_0\geq t_0$. 

\medskip

\q Note that if $t'=0$ then, putting $u=(r-t)/2$,  we have $u_0=(r_0-t_0)/2$ and $t+u_0=(r_0+t_0)/2<p-1$ so that  $(r,t)$ is not a counterexample, by Lemma 1.9. 

\q So we now assume  $t'>0$, and suppose  that we have a solution $(N,\sigma,\delta,I)$. First suppose $0\in I$, $\sigma_0= p-1-t_0$. Then we have $r_0\equiv -1+p-1-t_0+2\de_0$ (mod  $p$) so that $\de_0=(r_0+t_0)/2+1$ but then $\sigma_0+\de_0=p-1+(r_0-t_0)/2$ and the admissibility condition is violated.

\q So we now assume $t'>0$, $0\not\in I$, $\sigma_0=t_0+1$.  Hence   $r_0\equiv t_0+2\de_0$ (mod  $p$) and we get $\de_0=(r_0-t_0)/2$. In this case the condition $\sigma_0+\de_0\leq p-1$ is satisfied.

\q Writing $\sigma=\sigma_0+p\sigma'$, $\delta=\delta_0+p\delta'$ and $I'=\{i-1\vert 0\neq i\in I\}$  as usual we need to consider solutions of  the equations
$$t=(p^N-1)+t_0+1+p\sigma'-2p\sigma_{I'}'$$
i.e.,
$$t'-1=(p^{N-1}-1)+\sigma'-2\sigma_{I'}'  \eqno{(7)}$$
and
$$r\equiv (p^N-1)+t_0+1+(r_0-t_0)+p\sigma'+2p\delta' \hskip 20pt ({\rm mod} \, p^{N+1})$$
i.e.,
$$r'-1\equiv (p^{N-1}-1)+\sigma'+2\delta' \hskip 20pt ({\rm mod} \, p^N) \eqno{(8)}.$$

\q But one can solve   (7) and (8) uniquely for $(N-1,\sigma',\delta',I')$ by the minimality assumption and then $(N,\sigma_0+p\sigma',\delta_0+p\delta',I)$ is the unique solution for $(r,t)$.

\bigskip  \bs

Case C:   \  $r_0$, $t_0\neq p-1$ have the same parity (hence $r'$ and $t'$  have the same parity) and $r_0<t_0$.

\q We consider a solution $(N,\sigma,\delta,I)$ and, as usual, we assume $N>0$.

\q Assume that $0\not\in I$, $\sigma_0=t_0+1$.  We have $r_0\equiv -1+(t_0+1)+2\de_0$ (mod  $p$), which gives $\de_0=(r_0-t_0)/2+p$. But now $\sigma_0+\de_0=(r_0+t_0)/2+1+p$ and the admissibility condition is violated.

\q Hence we may assume $0\in I$, $\sigma_0=p-1-t_0$.   Then we have $r_0\equiv -1+p-1-t_0+2\de_0$ (mod $p$), giving $\de_0=(r_0+t_0)/2+1$. In this case we have
$$\sigma_0+\de_0=p-1-t_0+(r_0+t_0)/2+1=p-(t_0-r_0)/2\leq p-1$$
and the desired admissibility condition is satisfied.

\q Hence $(N,\sigma,\de,I)$ is a solution for $(r,t)$ if and only if 
$$t_0+pt'=(p^N-1)+p\sigma'-(p-1-t_0)-2p\sigma'_{I'}$$
i.e., 
$$t'=(p^{N-1}-1)+\sigma'-2\sigma'_{I'} \eqno{(9)}$$
and
$$r_0+pr'
\equiv (p^N-1)+(p-1-t_0)+p\sigma'+r_0+t_0+2+2p\de' \hskip 10pt \hbox{(mod $p^{N+1}$)}$$
i.e.,
$$r'-2\equiv (p^{N-1}-1)+\sigma'+2\de'  \hskip 10pt \hbox{(mod $p^{N}$).} \eqno{(10)}$$

\q Now since $r\geq t$ and $r_0<t_0$ we have $r'>t'$ and therefore $r'-2\geq t'$, since $r'$ and $t'$ have the same parity.  We can solve (9) and (10) uniquely for $(N-1,\sigma',\de',I')$ and then $(N,\sigma,\de,I)$ is the unique admissible quadruple for $(r,t)$.

\bigskip

Case D: \    $r_0$ and $t_0$ (and so $r'$ and $t'$) have different parities  and  $r_0+t_0\leq p-4$.

\bs

\q If $t'=0$ then, putting $u=(r-t)/2$ we have $2u_0\equiv r_0-t_0 \ ({\rm mod}\, p)$  and hence $u_0=(p+r_0-t_0)/2$ from which it follows that $u_0+t<  p-1$ and there is a unique solution by Lemma 1.9.

\q We assume from now on that $t'>0$. We consider a possible solution   $(N,\sigma,\delta,I)$ with $N>0$.  First suppose $0\in I$, $\sigma_0=p-1-t_0$.  Then 
$$r_0\equiv -1 + p-1 -t_0+ 2\de_0 \ ({\rm mod}\, p)$$
 so that $2\de_0\equiv r_0+t_0+2 \ ({\rm mod}\, p)$ and, by parity considerations, $2\de_0=r_0+t_0+2+p$. However, we require $\sigma_0+\de_0\leq p-1$, i.e. $2\sigma_0+2\de_0\leq 2p-2$, i.e. 
 $$2p-2-2t_0+r_0+t_0+2+p\leq 2p-2$$
 i.e.,   $r_0+t_0+2+p\leq 2t_0$ i.e., $(r_0-t_0)+p+2\leq 0$ and this is false. 
 
 \q We now suppose $0\not\in I$, $\sigma_0=t_0+1$.    We get $r_0\equiv -1 + t_0+1+2\de_0 \ ({\rm mod}\, p)$ so that $\de_0=(r_0-t_0+p)/2$.  We require  $\sigma_0+\de_0\leq p-1$,  i.e., $t_0+1+(r_0-t_0+p)/2\leq p-1$, i.e., $(r_0+t_0+p)/2\leq p-2$, i.e., $r_0+t_0\leq p-4$, and indeed this is the case.

\q Thus $(N,\sigma,\de,I)$ is a solution for $(r,t)$ if and only if 
$$t=(p^N-1)+t_0+1 +p\sigma'-2p\sigma_{I'}$$
i.e.,
$$t'-1=(p^{N-1}-1)+\sigma'-2\sigma_{I'}' \eqno{(11)}$$
and 
$$r_0+pr'\equiv (p^N-1)+t_0+1+p\sigma'+(r_0-t_0+p)+2p\delta' \hskip 20pt ({\rm mod} \, p^{N+1})$$
i.e., 
$$r'-2\equiv (p^{N-1}-1)+\sigma'+2\de'  \hskip 20pt ({\rm mod}\, p^N). \eqno{(12)}$$

\q Now we have $r'\geq t'$, $t'>0$ and $r'$ and $t'$ have different parities, from which it follows that $r'-2\geq t'-1\geq 0$. Hence by minimality there is a unique solution $(N-1,\sigma',\de',I')$ for $(r'-2,t'-1)$, and then $(N,\sigma,\de,I)$ is the unique solution for $(r,t)$. Hence $(r, t)$ is not a minimal counterexample.

\bigskip\bs

Case E:  \  $r_0, t_0\neq p-1$, $r_0$ and $t_0$ have different parities (and hence so do $r'$ and $t'$), and $r_0+t_0\geq  p-2$.

\bs

\q As usual we consider possible admissible solutions   $(N,\sigma,\de,I)$ with $N>0$. 

\q Suppose first that $0\not\in I$, $\sigma_0=t_0+1$. Then we have 
$$r_0\equiv -1 + t_0+1+2\de_0 \  ({\rm mod}\, p).$$
This gives $\de_0=(r_0-t_0+p)/2$. But then  
$$\sigma_0+\de_0= (r_0+t_0+p)/2+1\geq (2p-2)/2+1= p$$
 and the condition $\sigma_0+\de_0\leq p-1$ does not hold. Hence there is no such solution.
 
 \q Now suppose that $0\in I, \sigma_0=p-1-t_0$.  Then  we have $r_0\equiv -1+{p-1-t_0}+2\de_0 \ ({\rm mod}\, p)$ which gives  $\de_0=(r_0+t_0+2-p)/2$.   Thus we have 
 \begin{align*}\sigma_0+\de_0&=p-1-t_0+(r_0+t_0+2-p)/2=(r_0-t_0+p)/2\cr
 &\leq (r_0+p)/2\leq (p-2+p)/2=p-1
 \end{align*}
 and the desired condition $\sigma_0+\de_0\leq p-1$ holds.

   \q Thus $(N,\sigma,\de,I)$ is a solution for $(r,t)$ if and only if 

$$t_0+pt'=p^N-1+p\sigma'-(p-1-t_0)-2p\sigma_{I'}'$$
 i.e.,
$$t'=(p^{N-1}-1)+\sigma'-2\sigma_{I'}' \eqno{(13)}$$
and 
$$r_0+pr'\equiv p^N-1 + (p-1-t_0)+p\sigma'+(r_0+t_0+2-p)+2p\delta' \hskip 20pt ({\rm mod} \, p^{N+1})$$
i.e.,
$$r'-1\equiv p^{N-1}-1 +\sigma'+2\delta' \hskip 20pt ({\rm mod} \, p^N) \eqno{(14)}.$$

\q Since $r'$ and $t'$ have different parities we have $r' > t'$, and so $r'-1 \geq t'$. Then by minimality there is a unique admissible quadruple  $(N-1,\sigma,\delta',I')$ satisfying (13) and (14) and then  $(N,\sigma,\de,I)$ is the unique solution for  $(r,t)$. 

\q We have examined all cases and shown that there is no minimal counterexample in each case so the result is proved. \qed

\begin{example} The question of whether the Lie algebra $\Delta(2)$ appears as a direct summand of $Y(2n)$ has a special significance for the theory of reductive pairs. We here recover the result of Goodbourn,  \cite[Theorem 4.8]{Oliver},  describing this situation.

\q If $p=2$ then since $\Delta(2)$ is not a tilting module, it cannot occur as a direct summand of the tilting module $Y(2n)$. 

\q  We take $r=2n$, $s=2$ in Theorem 1.7$\phantom{\,}^\prime$, so that $u=n-1$. 

\q Assume $p=3$.  Then $s$ has standard form $2=(3^1-1)+0$, so that $N=1$, $\sigma=0$.  Thus $T(2)$ is a summand of $Y(2n)$ when $(n-1)_1\neq 2$, i.e., $n\equiv 1,2,3,4,5$ or $6$ modulo $9$. 

\q Assume $p>3$. Then $s$ has standard form $2=(p^0-1)+2$, so that $N=0$, $\sigma=2$.  Thus $T(2)$ is a summand of $Y(2n)$ when $(n-1)_0+2<p-1$, i.e., $(n-1)_0\neq  p-3,p-2,p-1$, i.e., $n$ is not congruent to $0,-1$ or $-2$ modulo $p$.

\q To complete the picture we take $p=2$ and determine when $T(2)$ is a summand of $Y(2n)$. We have the standard form $2=(2^1-1)+1$, so that $N=1$, $\sigma=1$.  Thus $T(2)$ is a summand of $Y(2n)$ when $(n-1)_0=0$, $(n-1)_1=0$, i.e., when $n\equiv 1$ modulo $4$.

\end{example}

\section{A short exact sequence}

\q  In this section, we consider a general module of the form $\nabla(r)\otimes \Delta(s)$, with $r\geq s>0$. Our method is to understand this module as a quotient of some $Y(m)$. The key result in making the passage to the quotient  is the short exact sequence described in  Proposition 2.1  below.

\q  We shall work with the action of the divided power operators $f_i$, $i\geq 0$, on rational $G$-modules.    For $t\in K$, we have the unipotent element 
$u(\xi)=\left( \begin{matrix} 1& 0\cr
\xi& 1 
\end{matrix}
\right)$ of $G$.   Let $U$ be the unipotent subgroup $\{u(\xi) \vert \xi\in K\}$ of $G$.  Suppose that $V$ is a rational $U$-module and that $v\in V$.  Then there  is a uniquely determined sequence of elements $v_0,v_1,\ldots$ such that only finitely many are non-zero and $u(\xi)v=\sum_{i=0}^\infty \xi^i v_i$, for all $\xi\in K$.  The divided powers operators $f_0, f_1,\ldots$ are elements of the algebra of distributions of $U$ satisfying $f_iv=v_i$, for $i\geq 0$. The action of the divided powers operators on a tensor product  $V\otimes V'$ of rational $U$-modules is given by
$$f_a(x\otimes y)=\sum_{a=b+c} f_bx\otimes f_cy$$
for $a\geq 0$ and $x\in V$, $y\in V'$.

\q   If $l_+$ is a non-zero highest weight vector of the Weyl module $\Delta(s)$, $s\geq 0$, then the elements $l_+=f_0l_+,  f_1 l_+, \ldots, f_s l_+$ form a $K$-basis of $\Delta(s)$ and $f_il_+=0$, for $i>s$.

\q Over a field of characteristic $0$, it follows from the usual Clebsch-Gordan formula that, for $r\geq s>0$, the module $\nabla(r)\otimes \Delta(s)$ is the direct sum of the modules $\nabla(r-s)$ and $\nabla(r+1)\otimes \Delta(s-1)$. However, the following weak version survives in arbitrary characteristic.

\begin{proposition} Suppose that $r\geq  s >0$. Then there exists a short exact sequence of $G$-modules
$$0\to \nabla(r-s)\to \nabla(r)\otimes \Delta(s)\to \nabla(r+1)\otimes \Delta(s-1)\to 0.$$
\end{proposition}

\begin{proof}   We write $x_1,x_2$ for the natural basis vectors of $E$, with $x_1$ having weight $1$ and $x_2$ having weight $-1$.  Multiplication in the symmetric algebra on $E$ gives the $G$-module homomorphism $\alpha: S^r E\otimes E\to S^{r+1}E$,  taking $x_1^ax_2^b\otimes y$ to $x_1^ax_2^by$, for $a,b\geq 0$, $a+b=r$, $y\in E$.

\q  We choose nonzero  highest weight vectors  $l_+$ in $\Delta(s)$ and  
$m_+$ in \\
$\Delta(s-1)$.  The dual of the  natural surjection $E\otimes S^{s-1}E\to S^sE$ is a $G$-module   embedding $\Delta(s) \to E\otimes \Delta(s-1)$, and it follows, by weight considerations,  that there is an embedding  $\beta: \Delta(s) \to E\otimes \Delta(s-1)$ taking $l_+$ to $x_1\otimes m_+$.

\q Note that, for $i>0$ we have 
$$\beta(f_il_+)=\sum_{i=u+v} f_ux_1\otimes f_v  m_+=x_1\otimes f_im_+ + x_2\otimes f_{i-1} m_+.$$

\q We define
$$\theta=(\alpha\otimes \id)\circ (\id\otimes \beta):\nabla(r) \otimes \Delta(s)\to \nabla(r+1)\otimes \Delta(s-1)$$
where the first $\id$ is the identity map on  $\Delta(s-1)$  and the second is the identity map on $\nabla(r)$.

\q We will show that $\theta$ is surjective, and that its kernel is isomorphic to $\nabla(r-s)$.  Note that, for $a,b\geq 0$ with $a+b=r$, we have 
$$\theta(x_1^ax_2^b\otimes l_+)=x_1^{a+1}x_2^b \otimes m_+ \eqno{(1)}$$
and, for $i\geq 1$, we have 
$$\theta(x_1^ax_2^b\otimes f_il_+)=x_1^{a+1}x_2^b\otimes f_im_+ + x_1^ax_2^{b+1}\otimes f_{i-1}m_+.\eqno{(2)}$$
 In particular,  we obtain
$$\theta(x_1^ax_2^b\otimes f_sl_+)=x_1^ax_2^{b+1}\otimes f_{s-1}m_+\eqno{(3)}$$
(since $m_+$ is a highest weight vector of $\Delta(s-1)$ and $f_sm_+=0$).

\q Hence , by (1),  the image ${\rm Im}(\theta)$ contains all elements $x_1^{a+1}x_2^b\otimes m_+$.   Furthermore, we have
\begin{align*}\theta(\sum_{i=1}^{s-1}&(-1)^{i-1}x_1^{i-1}x_2^{r-i+1}\otimes f_im_+)\cr
&=\sum_{i=1}^{s-1} (-1)^{i-1}  (x_1^i x_2^{r-i+1}\otimes f_im_+ + x_1^{i-1}x_2^{r-i+2}\otimes f_{i-1}m_+)
\end{align*}
by (2).
Note that this sum is telescopic and reduces to 
$$x_2^{r+1}\otimes m_+ +   (-1)^{s-2} x_1^{s-1}x_2^{r-s+2}\otimes f_{s-1} m_+.$$
But, from (3) we have that $x_1^{s-1}x_2^{r-s+2}\otimes f_{s-1} m_+\in {\rm Im}(
\theta)$ and therefore $x_2^{r+1}\otimes m_+ \in {\rm Im}(\theta)$. 

\q We have shown that $x_1^ax_2^b\otimes m_+\in {\rm Im}(\theta)$ for all $a,b\geq 0$ with $a+b=r+1$. Hence $\nabla(r+1)\otimes m_+\subseteq {\rm Im}(\theta)$. But now 
$$Z=\{ v\in \Delta(s-1) \vert \nabla(r+1)\otimes  v \subseteq {\rm Im}(\theta) \}$$
 is a $G$-submodule of $\Delta(s-1)$ containing the generator $m_+$ and hence 
 $Z=\Delta(s-1)$, i.e., ${\rm Im}(\theta)=\nabla(r+1)\otimes \Delta(s-1)$.
 
 \q Now $V=\nabla(r)\otimes \Delta(s)$ has a good filtration $0=V_0<  V_1 <  \cdots < V_{s+1}=V$ with $V_i/V_{i-1}\cong \nabla(r-s+2(i-1))$, for $1\leq i\leq s+1$. In particular $V_1$ is isomorphic to $\nabla(r-s)$. Also, the module $W=\nabla(r+1)\otimes \Delta(s-1)$  has a good filtration $0=W_0<W_1\cdots<W_s$ with $W_i/W_{i-1}\cong \nabla(r-s+2j)$, for 
$1\leq j\leq s$. But now, for $1\leq j\leq s$ then map $V_1\to W_j/W_{j-1}$ induced by $\theta$ is $0$ since the socle $L(r-s+2j)$ of $W_j/W_{j-1}$ is not a composition factor of $V_1$. Hence $\theta(V_1)=0$, i.e., $V_1$ is contained in  ${\rm Ker}(\theta)$. Now $\theta$ induces a surjective module homomorphism ${\overline \theta}: V/V_1\to W$ and since $V/V_1$ and $W$ have the same character and hence the same dimension, ${\overline \theta}:(\nabla(r)\otimes \Delta(s))/V_1\to  \nabla(r+1)\otimes \Delta(s-1)$ is an isomorphism, and we are done.
\end{proof}

\q We recall the following construction (for further details see, for example, \cite[Appendix]{q-Donk}).  Let $\pi$ be a set of non-negative integers. We say that a rational $G$-module $V$ belongs to $\pi$ if each composition factor $L$ of $V$ has the form $L(t)$ for some $t$ (depending on $L$) in $\pi$. Among all submodules of an arbitrary rational $G$-module $V$ there is a unique maximal one belonging to $\pi$, which we denote $O_\pi(V)$.   A set $\pi$ of non-negative integers will be called {\em saturated} if whenever $n\in \pi$ and $m\leq n$ with $n-m$ even then $m\in \pi$.  We have, e.g.  by  \cite[Proposition A.3.2(i)]{q-Donk}, the following result.

\begin{lemma}   If $\pi$ is a saturated set of non-negative integers  and $0\to V'\to V\to V''\to 0$ is a short exact sequence of $G$-modules with a good filtration then 
$$0\to O_\pi(V')\to O_\pi(V)\to O_\pi(V'')\to 0$$
 is exact.
 \end{lemma}

 Also we have, by e.g. by  \cite[Proposition A.3.1(ii) and Proposition A. 3.2]{q-Donk}, the following result.

 \begin{lemma} If $\pi$ is a saturated set of non-negative integers and $V$ is a finite dimensional $G$-module with a good filtration then:

 (i) $$(O_\pi(V):\nabla(r))= \begin{cases} (V:\nabla(r)),  &\hbox{ if  } r\in \pi; \cr
 0, & \hbox{ otherwise}
 \end{cases}$$

 (ii) $V/O_\pi(V)$ has a good filtration and 
  $$(V/O_\pi(V):\nabla(r))= \begin{cases}  0,   &\hbox{ if  } r\in \pi; \cr
 (V:\nabla(r)), & \hbox{ otherwise.}
 \end{cases}$$

 \end{lemma}

 \q It follows from Proposition 2.1 that, for  $r>s\geq 0$ the module $\nabla(r)\otimes \Delta(s)$ contains a unique submodule $M$, say, isomorphic to $\nabla(r-s)$ and $\nabla(r)\otimes \Delta(s)/M$ has a filtration with sections  $\nabla(t)$, with $r-s<t\leq r+s$, and $r+s-t$ even. We get the following consequences of Proposition 2.1.

 \begin{corollary}Let $r\geq s\geq0$. Then we have 
$$\nabla(r)\otimes \Delta(s)=Y(r+s)/O_\pi(Y(r+s))$$
 where $\pi=\{i\in \eno \vert i<r-s\}$.
\end{corollary}

\begin{corollary} Suppose that $r\geq s\geq 0$ and that $\pi$ is a saturated set of non-negative integers. Then  $(\nabla(r)\otimes \Delta(s))/O_\pi(\nabla(r)\otimes \Delta(s))$ is either zero or isomorphic to $\nabla(r')\otimes \Delta(s')$, for some $r'\geq s'\geq 0$.
\end{corollary}

\begin{proof}    We proceed by induction on  $s$.  If $s=0$ then 
$$O_\pi(\nabla(r)\otimes \Delta(s))=O_\pi(\nabla(r))$$
 which is either $0$ or $\nabla(r)$ so that $(\nabla(r)\otimes \Delta(s))/O_\pi(\nabla(r)\otimes \Delta(s))$ is either $\nabla(r)$ or $0$. 

\q We now assume that $s>0$ and that the result holds for smaller values.

If $r-s\not \in \pi$ then $O_\pi(\nabla(r)\otimes \Delta(s))=0$ and  again the result is clear. Assume now that $r-s\in \pi$. We identify $\nabla(r-s)$ with a submodule of $\nabla(r)\otimes \Delta(s)$. Thus we have 
\begin{align*}
(\nabla(r)&\otimes \Delta(s))/O_\pi(\nabla(r)\otimes \Delta(s))\cr
&\cong (\nabla(r)\otimes \Delta(s)/\nabla(r-s))/(O_\pi(\nabla(r)\otimes \Delta(s))/\nabla(r-s))\cr
&\cong (\nabla(r+1)\otimes \Delta(s-1))/O_\pi( \nabla(r+1)\otimes \Delta(s-1))
\end{align*}
and the result follows by induction. 
\end{proof}

\begin{corollary} If $C$ is a Clebsch-Gordan module then so is $C/O_\pi(C)$, for any saturated set of non-negative integers. 
\end{corollary}

\begin{proof}  We may assume that $C$ is indecomposable, and hence a direct summand of $\nabla(r)\otimes \Delta(s)$, for some $r\geq s$. Hence  $C/O_\pi(C)$ is a direct summand of 
$\nabla(r)\otimes \Delta(s)/O_\pi(\nabla(r)\otimes \Delta(s))$ and, by Corollary 1.6, this is $\nabla(r')\otimes \Delta(s')$, for some $r'\geq s'$, and hence $C/O_\pi(C)$ is Clebsch-Gordan.
\end{proof}

\section{The general result}

\q We  identify the indecomposable  Clebsch-Gordan modules, and give the promised decomposition of a general module of the form $\nabla(r)\otimes \Delta(s)$. We shall show here that the indecomposable Clebsch-Gordan modules are certain quotients of the indecomposable tilting modules.  To this end it is useful to note that these quotients are indecomposable.  This follows from the well-known observation recorded in Lemma 3.1  below. 

 \q In the proof we shall need Steinberg's tensor product. We write $G_1$ for the first infinitesimal subgroup of $G$. The modules $L(0),L(1),\ldots,L(p-1)$ form a complete set of pairwise non-isomorphic simple  $G_1$-modules. We write $F:G\to G$ for the usual Frobenius map and, for a rational $G$-module $V$ affording the representation $\pi:G\to \GL(V)$, write $V^{[1]}$  for the $K$-space $V$ regarded as a $G$-module via the representation $\pi\circ F$.  For $r\geq 0$ we  write $r=r_0+pr'$, with $0\leq r_0\leq p-1$, we have $L(r)=L(r_0)\otimes L(r')^{[1]}$, by Steinberg's tensor product theorem.

\begin{lemma} For $r\geq 0$, the tilting module $T(r)$ has simple head.

\end{lemma}

\begin{proof} The dual of a tilting module is tilting (and in fact all  tilting modules for ${\rm SL}_2(K)$ are self dual) so  this is equivalent to the statement that the indecomposable tilting modules have simple socle.   We shall use the  description of the tilting modules for ${\rm SL}_2(K)$   given in \cite{DTilt}.

\q   If $0\leq r\leq p-1$ then $T(r)$ is simple. In case $p\leq r\leq 2p-2$ the tilting module $T(r)$, as a module for $G_1$, is the injective hull of the simple module $L(2p-2-r)$. This has a simple $G_1$-socle and so certainly a simple socle as a $G$-module.

\q Now suppose that $r\geq 2p-1$. We write $r=m+pn$, with $p-1\leq m\leq 2p-2$ and $n>  0$. Then  according to \cite{DTilt} the module $T(r)$ may be realised as $T(m)\otimes T(n)^{[1]}$. The $G_1$-socle, and hence the $G$-socle of $T(m)$ is $L(a)$, where $a=2p-2-m$. Hence if, for $s=b+pc$, with $0\leq b\leq p-1$, $c\geq 0$, the module $L(s)$ appears in the socle of $T(r)$ then we have $b=a$.  The multiplicity of $L(s)$ in the $G$-socle is, by Schur's Lemma, the dimension of $\Hom_G(L(s),T(r))$.  Now we have
\begin{align*}\Hom_G(&L(s),T(r))=\Hom_{G_1}(L(s),T(r))^G\cr
&=\Hom_{G_1}(L(a)\otimes L(c)^{[1]},T(m)\otimes T(n)^{[1]})^G\cr
&=(\Hom_{G_1}(L(a),T(m))\otimes \Hom_K(L(c)^{[1]},T(n)^{[1]}))^G.
\end{align*}

\q Now $T(m)$ has $G_1$-socle $L(a)$ so that $\Hom_{G_1}(L(a),T(m))$ is one dimensional and hence trivial as a $G$-module. Hence we have 
\begin{align*}\Hom_G(L(s),T(r))&= \Hom_K(L(c)^{[1]},T(n)^{[1]}))^G\cr
&=\Hom_G(L(c)^{[1]},T(n)^{[1]})=\Hom_G(L(c),T(n)).
\end{align*}
We may assume inductively that $T(n)$ has simple socle $L(d)$, for some $d\geq 0$ and so we obtain
$$\Hom_G(L(s),T(r))=\begin{cases}K, & \hbox{ if } s=a+pd;\cr
0, & \hbox{ otherwise}
\end{cases}$$
which proves that $T(r)$ has simple socle $L(a+pd)$.
\end{proof}

\q Let $r\in \eno$, written in standard form $r=(p^N-1)+\sigma$. We put $S(r)=\{i\in \eno\vert i<N, \sigma_i\neq 0\}$. The sections in a good filtration of $T(r)$ are, by Proposition 1.3,  the modules $\nabla(r-2\sigma_I)$.   We define a total order on the power set of $S(r)$ by the condition $I\varleq J$ if $\sigma_I\leq \sigma_J$.  Thus we have $r-2\sigma_I\leq r-2\sigma_J$ if and only if $J\varleq I$.

\begin{remark} In fact, it is easy to see that $I\varleq J$ if and only if $\max(I\backslash J)\leq \max(J\backslash I)$ (where $\max(A)$ denotes the maximum of a finite set of non-negative integers $A$).
\end{remark}
\bs

\q Let $I\subseteq S(r)$ and put $\pi(I)=\{i\in \eno\vert i<r-2\sigma_I\}$. We define $T(r)_I=T(r)/O_{\pi(I)}(T(r))$.

\begin{lemma}  (i)  $O_{\pi(I)}(T(r))$ has a good filtration with sections $\nabla(r-2\sigma_J)$, for $r-2\sigma_J<r-2\sigma_I$, i.e., $I\neq J$, $I \varleq J$; 

(ii) $T(r)_I$ has a good filtration with sections $\nabla(r-2\sigma_J)$, for $r-2\sigma_J\geq r-2\sigma_I$, i.e.,  $J\varleq I$; and 

(iii) $T(r)_\emptyset=\nabla(r)$ and $T(r)_{S(r)}=T(r)$. 

\end{lemma}

\begin{lemma}  For $r,s\in \eno$ and $I\subseteq S(r)$, $J\subseteq S(s)$ the modules $T(r)_I$ and $T(s)_J$ are isomorphic if and only if $r=s$ and $I=J$.

\end{lemma}

\begin{proof}   The module $T(r)_I$ unique highest weight $r$ and the module $T(s)_I$ has unique highest weight $s$ so that if $T(r)_I$ and $T(s)_J$ are isomorphic then $r=s$. Now $T(r)_I$ has section $\nabla(r-2\sigma_I)$ and if $\nabla(r-2\sigma_J)$ is also a section then we have $J\varleq I$.  Thus if $T(r)_I$ and $T(r)_J$ are isomorphic then $J\varleq I$, and by symmetry $I\varleq J$. Hence $I=J$.
\end{proof}

We summarise our findings in our main result.

\begin{theorem} Let $r\geq s\geq 0$. Then each  indecomposable summand of $\nabla(r)\otimes \Delta(s)$ occurs at most once in a direct sum decomposition. Each indecomposable summand has the form $T(m)_I$, for some $m\geq 0$, which we write in standard form  $m=(p^N-1)+\sigma$,  and where  $I$  is a subset of  $S(m)$.  Furthermore, for such $m$ and $I$, we have $(\nabla(r)\otimes \Delta(s) \vert T(m)_I)\neq 0$ if and only if $(Y(r+s)\vert T(m))\neq 0$ (and the condition for this is described in Theorem 1.7), $m-2\sigma_I\geq r-s$ and $\sigma_I\geq \sigma_J$ for every $J \subseteq S(m)$ such that $ m-2\sigma_J \geq r-s$.
\end{theorem}

\begin{proof} The first two statements are true by virtue of  Remark 1.1 and Lemma 3.4.  We decompose $Y(r+s)$ as $\bigoplus_{m\in M} T(m)$. The set $M$ is described  by Theorem 1.7.  Let $\pi=\{i\in \eno \vert i<r-s\}$.  Then we have
$$\nabla(r)\otimes \Delta(s)=Y(r+s)/O_\pi(Y(r+s))=\bigoplus_{m\in M} T(m)/O_\pi(T(m)).$$
Let $m\in M$.  Suppose $T(m)/O_\pi(T(m)) \neq 0$. Then $(T(m)/O_\pi(T(m)):\nabla(m))\neq 0$ and hence $m\not\in \pi$, i.e., $m\geq r-s$. We choose $I\subset S(m)$ such that $m-2\sigma_I\geq r-s$ and $\sigma_I\geq \sigma_J$ for every $J \subseteq S(m)$ such that $ m-2\sigma_J \geq r-s$. Then $m-2\sigma_I \leq m-\sigma_J$ for all $J\in S(m)$ such that $m-2\sigma_J \geq r-s$. We then have $T(m)/O_\pi(T(m)) =T(m)_I$, as required.

\end{proof}

\q Combining this with Lemma 3.4 we have the following result.

\begin{corollary}   The modules $T(r)_I$, for $r\in \eno$, $I\subseteq S(r)$, form a complete set of pairwise indecomposable non-isomorphic Clebsch-Gordan modules.

\end{corollary}

\begin{example} We ask when the trivial module  $T(0)$ is a summand of $\nabla(r)\otimes \Delta(s)$ (with $r\geq s\geq 0$).  Then we have $0\geq r-s$ so that $r=s$.   The standard expression for $0$ is $0=(p^0-1)+0$  so we have $N=0$, $\sigma=0$. So, putting $u=(r+s-0)/2=r$, we have $r_0<p-1$, i.e., $s \not\equiv 0$ (modulo $p$). This is the condition for $\nabla(r)=S^rE$ to have dimension prime to $p$, so this is a special case of the result of Benson and Carlson, \cite[Theorem 3.1.9]{Benson}. 
\end{example}

\begin{example} We ask when the natural  module $E=T(1)$ is a summand of $\nabla(r)\otimes \Delta(s)$, for $r\geq s$. For this we require $1\geq  r-s$ so that $r=s+1$ and $T(1)$ is a direct summand of $Y(2s+1)$.  

\q  If $p=2$ the we have the standard expression $1=(2^1-1)+0$ so that $N=1$, $\sigma=0$ and, putting $u=(r+s-1)/2=s$, the condition is $s_1<1$, i.e., $s+1=0$, i.e., $s \equiv 0$ or $1$ modulo $4$.

\q If $p>2$ we have the standard expression $1=(p^0-1)+1$ so that $N=0$ and $\sigma=1$.  Again putting $u=(r+s-1)/2=s$ the condition is now $s_0+1\leq p-1$ and $s_1<p-1$, i.e., $s_0\not\equiv -2,-1$ (modulo $p$) and $s_1\not\equiv -1$ (modulo $p$).
\end{example}

\begin{example}  Suppose that $\nabla(t)$ is a direct summand of of $\nabla(r)\otimes \Delta(s)$, with $r\geq s$. Then we have $t\geq r-s$  with $r+t-s$ even. We write $t$ in its standard expression $t=(p^N-1)+\sigma$ 
and set $u=(r+s-t)/2$.  If $\sigma\neq 0$ we require $\sigma-p^m\sigma_m<r-s$, where $\sigma_m\neq 0$ and $\sigma_i=0$ for $i<m$. To summarise, with the above notation, $\nabla(t)$ is a summand of $\nabla(r)\otimes \Delta(s)$ if and only if $t\geq r-s$, $r+s-t$ is even, $\sigma_i+u_i\leq p-1$ for $0\leq i<N$, $\sigma_N+u_N<p-1$ and if $\sigma\neq 0$ then $\sigma-p^m\sigma_m<r-s$.
\end{example}

 \end{document}